\documentclass[a4paper,12pt]{amsart}
\usepackage{amssymb}
\usepackage[matrix,arrow,curve]{xy}
\author{Olga V. Chuvashova}
\address{
Department of Higher Algebra\\
Faculty of Mechanics and Mathematics\\
Moscow State University\\
119992 Moscow, Russia}
\email{chuvashova@gmail.com}

\author{Nikolay A. Pechenkin}
\address{
Algebraic Geometry Section\\
Steklov Mathematical Institute\\
Russian Academy of Sciences\\
Gubkin str. 8, GSP-1\\
119991 Moscow, Russia}
\email{kolia.pechnik@gmail.com}

\thanks{The second author was partially supported by the grants NSh-5139.2012.1, MK-6612.2012.1, RFFI 12-01-31506, RFFI 12-01-33024}

\title{The moduli stack of affine stable toric varieties}



\newcommand{\Sp}{{\rm Spec\ }}

\newcommand{\Isom}{{\rm Isom}}
\newcommand{\Oo}{\mathcal{O}}

\renewcommand{\H}{\mathcal{H}}
\newcommand{\cQ}{\mathcal{Q}}

\newcommand{\Q}{\mathbb{Q}}
\newcommand{\A}{\mathbb{A}}
\newcommand{\X}{\mathbb{X}}
\newcommand{\T}{\mathbb{T}}
\renewcommand{\P}{\mathbb{P}}
\newcommand{\Z}{\mathbb{Z}}
\newcommand{\N}{\mathbb{N}}
\newcommand{\lsp}{\textbf{Spec}}
\newcommand{\cone}{{\rm cone}}

\newcommand{\G}{\mathcal{G}}

\newcommand{\Chi}{\mathfrak{X}}
\newcommand{\F}{\mathcal{F}}
\newcommand{\M}{\mathcal{M}}

\renewcommand{\Im}{\text{Im}\,}

\newtheorem{theorem}{Theorem}[section]

\newtheorem{lemma}[theorem]{Lemma}
\theoremstyle{definition}
\newtheorem{definition}[theorem]{Definition}
\newtheorem{example}[theorem]{Example}

\theoremstyle{remark}
\newtheorem{remark}[theorem]{Remark}

\begin{document}

\begin{abstract}
Let $\X$ be an irreducible affine $T$-variety. We consider families of affine stable toric $T$-varieties over $\X$ and give a description of the corresponding moduli space as the quotient stack of an open subscheme in a certain toric Hilbert scheme under the action of a torus.
\end{abstract}

\maketitle

\section{Introduction}

Let $T$ be an algebraic torus. By a $T$-variety ($T$-scheme) we, as usual, mean a variety (scheme) with a (not necessarily effective) action of $T$. It is called toric if $T$ acts with dense open orbit. The notion of a stable toric $T$-variety was introduced by Alexeev in 2002 \cite{Al}. It provides a generalization of the notion of the toric $T$-variety to some reducible varieties which are unions of toric varieties, glued along toric subvarieties. Stable toric varieties are analogs of stable curves in the case of toric varieties. In the affine case its definition can be formulated in a pretty simple way (see Definition~\ref{astv}).

In this work we use the notion of stable toric variety to consider a version of a moduli space parameterizing general $T$-orbits closures and some their degenerations for the action of $T$ on the irreducible affine variety $\X$. We consider a category $\M_{\X,T}$ of the families of affine stable toric $T$-varieties over $\X$. We show that objects in $\M_{\X,T}(k)$ possess finite automorphism groups (non-trivial in general), so in general $\M_{\X,T}$ can not be represented by a scheme. Then one may ask if it can be represented by a stack. The answer to this question is positive. Our main result is Theorem~\ref{main}, which states that $\M_{\X,T}$ is a quotient stack.

The proof of the main theorem is constructive. For each irreducible affine $T$-variety $\X$ we choose a finite subset $A$ in the saturation of its weight monoid. Here $A$ is required to satisfy some technical conditions. Then we consider the $T$-variety $\X\times \A^d$, where $d=\#A$ and the action of $T$ on $\A^d$ is given by the weights in $A$. There is also an action of $d$-dimensional torus $\T$ on $\X\times \A^d$ given by its diagonal action on $\A^d$. It induces the action of $\T$ on the toric Hilbert scheme $H_{\X\times\A^d, T}$. We provide explicit conditions that define an open $\T$-invariant subscheme $\widetilde{H}\subset H$. The precise statement of Theorem~\ref{main} is that there is an equivalence of the categories fibred in groupoids $\M_{\X,T}\cong [\widetilde H/\T]$.

\medskip

The similar questions about stable varieties and their moduli spaces were investigated by Alexeev and Brion in \cite{Al}, \cite{AB1}, \cite{AB}, \cite{AB2}. For example, in \cite{AB2} the moduli space of stable spherical varieties over a projective space is constructed. For a detailed review of the history of this question we refer
to the introduction of \cite{AB2}.

\medskip

This paper is organized as follows. In Section~\ref{int} we briefly recall necessary information on $T$-varieties, GIT-fans and toric Hilbert schemes. In Section~\ref{dastv} we provide definitions of affine stable toric varieties and families of affine stable toric varieties. Then we prove several technical statements analogues to those of \cite{AB2}. In Section~\ref{mod} we investigate the category of affine stable toric varieties, formulate and prove the main result of this paper
and provide a detailed example.

\section{Preliminaries}\label{int}

Fix the base field $k$ to be an algebraically closed field of characteristic zero. We work in the category of schemes over $k$ everywhere except the last section, where we need to expand it to the category of algebraic stacks over $k$. Our references are \cite{EH}, \cite{Har} on schemes, and \cite{Go}, \cite{LM} on stacks. A variety is a separated reduced scheme of finite type.

\subsection{Basic facts about T-varieties}

Let $T:=(k^{\times})^n$ be an algebraic torus, and $X$ be an affine $T$-variety. Then its algebra of regular functions $k[X]$ is graded by the group $\Chi(T)$ of characters of
$T$  $$k[X]=\bigoplus_{\chi\in\, \Chi(T)} k[X]_{\chi},$$
 where  $k[X]_{\chi}$ is the
subspace of $T$-semiinvariant functions of weight $\chi$. The \emph{weight set} is defined as
$$\Sigma_{X}:=\{\chi\in \Chi(T)\ : \ k[X]_\chi\ne 0\}.$$ If $T$ acts on $X$ faithfully then $\Sigma_X$ generates the group $\Chi(T)$. Denote by $\widehat{\Sigma}_X:=\{\chi\in\Chi(T)\text{ such that }n\chi\in\Sigma_X\text{ for some }n\in \Z_{>0} \}$ the \emph{saturation of} $\Sigma_X$.

Assume that $X$ is irreducible. Then $\Sigma_X$ formes a finitely generated monoid so we call it the \emph{weight monoid}. It is easy to see that in this case $\widehat{\Sigma}_X=\sigma_X \cap \Chi(T)$, where $\sigma_X:=\text{cone}(\Sigma_X)\subset \Chi(T)_{\Q}:=\Chi(T)\otimes_{\Z}\Q$ is the \emph{weight cone}.

Recall from \cite{BH} the notion of a GIT-fan. For any $x\in
X$, the {\it orbit cone} $w_x$ associated to $x$ is the following
convex cone in $\Chi(T)_{\Q}$
$$w_x:=\cone\{\chi\in \Sigma : \exists\ f\in k[X]_{\chi} {\rm\ such \ that \ } f(x)\ne
0\}=\Sigma_{\overline{Tx}}.$$
For any character $\chi\in \Sigma$, the associated {\it
GIT-cone}~$\sigma_\chi$ is the intersection of all orbit cones
containing $\chi$
 $$\sigma_\chi=\bigcap_{\{x\in X : \chi\in
w_x\}}w_x.$$

Recall that a \emph{quasifan} $\Lambda$ in $\Chi(T)_{\Q}$ is a finite collection of (not necessarily pointed) convex, polyhedral
cones in $\Chi(T)_{\Q}$ such that for any $\lambda\in\Lambda$ also all faces of $\lambda$ belong to
$\Lambda$, and for any two $\lambda, \lambda'\in\Lambda$
 the intersection $\lambda\cap\lambda'$
is a face of both, $\lambda$ and $\lambda'$. A
quasifan is called a \emph{fan} if it consists of pointed cones. The \emph{support} $|\Lambda|$ of a quasifan $\Lambda$ is the union of its cones.

By \cite[Theorem~2.11]{BH}, the collection of GIT-cones forms a fan
$\cQ_X$ having $\sigma_X$ as its support.

\subsection{The toric Hilbert scheme}

In this section we give a short overview on the toric Hilbert schemes. For the details we refer reader to \cite{B}, \cite{CP}, \cite{PS}.

\begin{definition}
 Given a function $h:\Chi(T)\to \N$, a \emph{family of affine} $T$\emph{-schemes over a base} $S$ \emph{with Hilbert function} $h$ is a $T$-scheme $Z$ equipped with an affine, $T$-invariant morphism of finite type $p: Z\to S$ such that for all $\chi\in \Chi(T)$ the sheaf $p_*(\Oo_Z)_\chi$ is locally free of rank $h(\chi)$.
\end{definition}

Note that the morphism $p$ here is automatically flat. If $h(0)=1$ then $p$ is a good quotient of $Z$ by the action of $T$ (see \cite[Section~2.3]{ADHL} for the definition of good quotients).

\medskip

The functions $h$ we will consider are of the form
$$h=h_\Sigma(\chi):= \left\{ \begin{array}{rl}
 1 & \mbox{if}\ \ \chi\in \Sigma, \\
 0 & \mbox{otherwise.}
\end{array}\right. $$
for the finitely generated monoid $\Sigma\subset \Chi(T)$. Fix an irreducible affine $T$-variety $X$.

\begin{definition}  The {\it toric Hilbert functor}
is the contravariant functor $\H_{X, T}$ from the category of schemes to the category of sets
assigning to any scheme $S$ the set of all closed $T$-stable subschemes $Z\subseteq S\times X$ such that
the projection $p:Z\to S$ is a family of affine $T$-schemes with Hilbert function $h_{\Sigma_X}$.
\end{definition}

The toric Hilbert functor $\H_{X, T}$ is represented by the quasiprojective scheme $H_{X, T}$ called the \emph{toric Hilbert scheme}. We will denote $U_{X, T}$ the universal family over $H_{X, T}$. For any $Z\in \H_{X, T}(S)$ we have $Z=S\times_{H_{X, T}}U_{X, T}$.

\section{Affine stable toric varieties: definitions and properties}\label{dastv}

\begin{definition}\label{astv}
By an affine (multiplicity-free) {\it stable toric variety} under an action of an algebraic torus $T$ we mean an affine $T$-variety $X$ such that $k[X]$ is multiplicity-free as a $T$-module and $\Sigma_X=\widehat{\Sigma}_X$.
\end{definition}

Recall that a module $V$ is multiplicity-free if every irreducible submodule of $V$ occurs with multiplicity one. For affine stable toric variety $X$ it means that $\dim k[X]_{\chi}=1$ for all $\chi \in \Sigma_X$.

One can show that this definition coincides with the general definition of a stable toric variety given in \cite[Section~1.1.A]{Al} (see \cite[Lemma~2.3]{AB1}). If $X$ is irreducible, then it is a (normal) toric $T$-variety. The easiest example of reducible affine stable toric variety is a cross, which is defined as a spectrum of $k[x,y]/(xy)$, where the action of $T=k^\times$ is given by grading $\deg(x)=1$, $\deg(y)=-1$. More generally, for any quasifan $\Lambda$ in $\Chi(T)$ one can define the quasifan algebra $k[\Lambda]$ in the following way
 $$k[\Lambda]:=\bigoplus_{\chi \in |\Lambda| \cap \Chi(T)} k\cdot x^{\chi},
\qquad
x^{\chi_1}\cdot x^{\chi_2} :=
\left\{
\begin{array}{ll}
x^{\chi_1+\chi_2} & \text{ if there exist }\lambda \in  \Lambda \text{ such that }\chi_1,\,\chi_2 \in \lambda, \\
0           & \text{ else.}
\end{array}
\right.
$$
 Then $X[\Lambda]:=\Sp k[\Lambda]$ is an affine stable toric $T$-variety. These are called toric bouquets in \cite[Definition~7.2]{AH} in the case when $\Lambda$ is the normal fan of some polyhedron. The complete description of affine stable toric varieties comes from this construction by adding a cocycle (see \cite[Lemma~2.3.11 and Theorem~2.3.14]{Al} for details).

%\begin{proposition}\cite[Proposition~3.3]{AB2} There are only finitely many $T$-orbits in $X$. Let $C(X)$ be the set of their closures, partially ordered by inclusion. Then any $Y\in C(X)$ is a normal toric variety and $$X=\underrightarrow{\lim}_{Y\in C(X)}Y.$$
%\end{proposition}

%We shall define the category $\S_{\X, T}$ of affine stable toric varieties over $\X$.
\medskip

From now on let $\X$ denote an irreducible affine $T$-variety.

\begin{definition}\label{family}
A {\it family of affine stable toric} $T$-{\it varieties over} $\X$ \emph{with a base} $S$ is a $T$-scheme $X$ equipped with two morphisms $p_S$ from $X$ to some scheme $S$, and $p_\X: X\to \X$ such that

\begin{enumerate}
\item $p_\X$ is $T$-equivariant;
\item $p_S$ is a family of affine $T$-schemes with Hilbert function $h_{\widehat{\Sigma}_{\X}}$;
\item geometric fibers of $p_S$ are all reduced;
\item the morphism $p_{S\times \X}:=p_S\times p_\X : X\to S\times \X$ is finite.
\end{enumerate}
\end{definition}

Any geometric fiber $X_{s}$ of such family $X$ is an affine stable toric variety for the torus $T$.
If all  $X_{s}$ are irreducible (that is, $X$ is a family of (normal) affine toric varieties), then $X$ is locally trivial over $S$ (see \cite[Lemma~7.4]{AB1}).

%\begin{example}

\bigskip

For a family $p_S: X\to S$ of affine stable toric $T$-varieties over $\X$, denote by $\F_\chi$ the invertible sheaf of $\Oo_S$-modules $(p_S)_*(\Oo_X)_\chi$, $\chi\in\widehat{\Sigma}_{\X}$ :
$$\F:=(p_S)_*\Oo_X=\bigoplus_{\chi\in \widehat{\Sigma}_{\X}} \F_\chi.$$
For any $\chi, \chi'\in \widehat{\Sigma}_{\X}$, we have the multiplication map $$m_{\chi, \chi'}:\F_\chi\otimes_{\Oo_S}\F_{\chi'}\to \F_{\chi+\chi'}.$$
Since the fibers of $p_S$ are reduced, the $N$th power maps $\F_\chi^N\to \F_{N\chi}$ are isomorphisms for all $\chi\in \widehat{\Sigma}_{\X}$.

%Then there exists a decomposition of $\cone(\Sigma)$ into a finite set of cones such that   the characters $\chi$ and $\chi'$ lie in
% the same cone of this decomposition if and only if the multiplication map $$m_{\chi, \chi'}:F_\chi\otimes_{\O_S}F_{\chi'}\to F_{\chi+\chi'}$$ is nonzero. In this case $m_{\chi, \chi'}$ is an isomorphism (in particular, the $N$th power maps $F_\chi^N\to F_{N\chi}$ are isomorphisms for all $\chi\in \Sigma$). The weight set $\Chi^T_Y$ of any $T$-subvariery $Y$ of any geometric fiber of $p_S$ is the intersection of $\Chi(T)$ with some cone of this decomposition.

The following lemmas are analogues of \cite[Lemma~4.7]{AB2} and \cite[Lemma~4.8]{AB2} in our affine settings.

\begin{lemma}\label{lem1}
 The following properties hold for any family $p_S: X\to S$ of affine stable toric $T$-varieties over $\X$:
\begin{enumerate}
\item The weight cone $\sigma_Y$ of any $T$-subvariety $Y$ of any geometric fiber $X_{s}$ is a union of GIT-cones $\sigma\in \cQ_{\X}$.

\item For any GIT-cone $\sigma\in \cQ_{\X}$ and $\chi, \chi'\in \sigma\cap\Chi(T)$ the map $m_{\chi, \chi'}$ is an isomorphism.
\end{enumerate}
\end{lemma}

\begin{proof}
 (1) The weight cone of any $T$-subvariety of $X_s$ is a union of weight cones of its irreducible components. Note that any normal toric $T$-variety $Y$ equipped with a finite T-equivariant morphism $\phi:Y\to \X$ is uniquely determined by its image $\phi(Y)$, and the weight cones $\sigma_Y$ and $\sigma_{\psi(Y)}$ coincide. By definition, the minimal $GIT$-cone containing a given character $\chi$ is the intersection of the weight cones $\sigma_{\overline{Tx}}$ of all orbit closures $\overline{Tx}\subset \X$ such that $\chi\in \sigma_{\overline{Tx}}$. Thus the statement is clear.

 (2) Let $s\in S$ be a geometric point. Then $X_s$ is an affine stable toric $T$-variety. Note that $\F_s=k[X_s]\otimes (\Oo_S)_s$ and recall that $k[X_s]$ is multiplicative-free. Now one can use the restriction on irreducible component of $X_s$ and the first part of this lemma to show that $(m_{\chi, \chi'})_s$ is an isomorphism. It follows that $m_{\chi, \chi'}$ is an isomorphism.
\end{proof}

The finite $T$-equivariant morphism $p_{S\times \X}: X=\textbf{Spec}_S \F \to S\times \X=\textbf{Spec}_S (\Oo_S\otimes k[\X])$ corresponds to the homomorphism $\varphi: S\otimes k[\X] \to \F$ of graded $\Oo_S$-algebras, which makes $\F$ into a finite $\Oo_S\otimes k[\X]$-module.

\begin{lemma}\label{lem2}
\begin{enumerate}
 \item There exists a finite subset $A\subset \widehat{\Sigma}_{\X}$ such that the $\Oo_S$-algebra $\F:=(p_S)_*\Oo_X$ is generated by $\F_\chi$, $\chi\in A$ for any family $p_S:X\to S$ of affine stable toric $T$-varieties over $\X$.

 \item There exists a positive integer $N$, which depends only on the $T$-variety $\X$, such that  the map  $\varphi_{N\chi}: \Oo_S\otimes k[\X]_{N\chi}\to \F_{N\chi}$ is surjective for any family $p_S:X\to S$ of affine stable toric $T$-varieties over $\X$ and for any $\chi\in\widehat{\Sigma}_{\X}$.
\end{enumerate}
\end{lemma}

\begin{proof}
(1) Let $A\subset \widehat{\Sigma}_{\X}$ be a finite subset such that the monoid $\widehat{\Sigma}_{\X}\cap \sigma$ is generated by $A\cap \sigma$, for any $\sigma\in\cQ_X$. By Lemma \ref{lem1} (2), it follows that $A$ satisfies the statement of the lemma.

(2)  We use notations $\F^{(\chi)}$ and $k[\X]^{(\chi)}$ for the sheaf of $\Oo_S$-algebras $\bigoplus_{n=0}^{\infty} \F_{n\chi}$ and the algebra $\bigoplus_{n=0}^{\infty} k[\X]_{n\chi}$ respectively. There exists a positive integer $n_\chi$ such that $k[\X]^{(n_\chi\chi)}$ is generated by its component of degree one. By weight considerations we see that $\F^{(n_\chi\chi)}$ is a finite $k[\X]^{(n_\chi\chi)}$-module. Since the property of morphism to be surjective is local, we can assume that the sheaf of $\Oo_S$-algebras $\F^{(n_\chi\chi)}$ is isomorphic to the polynomial ring $\Oo_S[z]$. If $z\not\in \Im\varphi_{n_\chi\chi}$, then $\varphi (\F^{(n_\chi\chi)})=\Oo_S$ what is in contradiction with the finiteness condition. So, $z\in \Im\varphi_{n_\chi\chi}$ and $\varphi_{n_\chi\chi}$ is surjective.

The least common multiple $N$ of $n_\chi$, $\chi\in A$, satisfies the statement of the lemma.
\end{proof}

\begin{remark}\label{fin}
In the proof of part (2), in fact, we deduce a criterion of the finiteness of the morphism $p_{S\times\X}$ for arbitrary families $p_S: X\to S$ of affine $T$-schemes with Hilbert function $h_{\widehat{\Sigma}_{\X}}$ equipped with the $T$-equivariant morphism $p_{\X}: X\to \X$. Indeed, choose $A\subset\widehat{\Sigma}_{\X}$ and $N\in \N$ satisfying the conditions of Lemma~\ref{lem2}. Then $p_{S\times \X}$ is a finite morphism if and only if the homomorphism $\Oo_S\otimes k[\X]_{N\chi}\to (p_S)_{*}(\Oo_X)_{N\chi}$ is surjective for all $\chi\in A$.
\end{remark}

%\begin{remark} \label{rmk2}
%Given a family $p_S:X\to S$ of affine stable toric $T$-varieties over $\X$, we shall denote by $\Q(\overline s)$ the set of weight cones
%$\cone(\Chi^T_Y)$ of all irreducible $T$-subvarieties $Y$ of the geometric fiber $X_{\overline s}$; so $\Q(\overline s)$ is a fan with support $\cone(\Sigma)$. By Lemma \ref{lem1}, it follows that any cone $\sigma$ in $\Q(\overline s)$ is a union of GIT-cones (in particular, there exists only finitely many possible fans $\Q(\overline s)$), and by Lemma \ref{lem2}, it follows that for any cone $\sigma\in \Q(\overline s)$, the monoid $\sigma \cap \Chi(T)$ is generated by $\sigma\cap A$.

%Further, for any fan $\mathcal C$ with support $\cone(\Sigma)$ we have an open subscheme $S_{\mathcal C}$ of $S$ consisting of points $s\in S$ such that the fan $\mathcal C$ refines $\Q(\overline{s'})$.

% Further, any $\Q(\overline s)$ denote by  $S_{\Q(\overline s)}$ consisting of points $s'\in S$ such that the fan $\Q(\overline{s})$ refines $\Q(\overline{s'})$. In particular, the open subsets  $S_s$ for those $s\in S$ that the fan $\Q(\underline s)$ is coarest, form
%\end{remark}

\section{Moduli stack of affine stable toric varieties}\label{mod}

From now on we fix a subset $A=\{\chi_1,\ldots,\chi_d\}\subset \widehat{\Sigma}_{\X}$, where $d:=\#A$, and a positive integer $N$ satisfying the conditions of Lemma \ref{lem2}.
We consider the category of stable toric varieties over $\X$ and prove that it is a quotient stack.

\begin{definition}
The {\it category} $\M_{\X, T}$ {\it of affine stable toric varieties} over $\X$ is the category whose objects are families of affine stable toric $T$-varieties over $\X$. For two families $p_{S}:X\to S$ and $p'_{S}: X'\to S'$ in $\M_{\X, T}$, a morphism $\phi$ is a pair consisting of a morphism $ \phi_S: S\to S'$ and a $T$-equivariant morphism $\phi_X:X\to X'$ such that the following diagrams are commutative:

$$
\xymatrix{
X \ar[rr]^{\phi_X}\ar[dd]_{p_S} && X' \ar[dd]^{p'_S} && X \ar[rr]^{\phi_X}\ar[ddr]_{p_{\X}} && X' \ar[ldd]^{p'_{\X}} \\
&&& , \\
S \ar[rr]_{\phi_S} && S' &&& \,\,\X\, . \\
}
$$
\end{definition}

\begin{remark} \label{rmk}
Note that the first diagram is cartesian. Indeed, the fibered product $S\times_{S'} X'$ is a family of affine stable toric $T$-varieties over $\X$ with base $S$, and we have a morphism $X\to S\times_{S'} X'$ in $\M_{\X, T}$. So we have only to show that for any families $X, X'$ in  $\M_{\X, T}$ over the same base $S$ and any morphism $\phi=({\rm Id}_S, \phi_X)$ from $X$ to $X'$ in $\M_{\X, T}$, the morphism $\phi_X$ is an isomorphism. Denote $\F:=(p_S)_*{\Oo_X}$ and $\F':=(p'_S)_*{\Oo_{X'}}$. The $T$-equivariant morphism $\phi_X$ corresponds to the homomorphism $\alpha: \F'\to\F$ of graded $\Oo_S$-algebras. Then, by Lemma \ref{lem2} (2), it follows that for any $\chi\in \widehat{\Sigma}_{\X}$, the morphism of invertible sheaves of  $\Oo_S$-modules  $\alpha_{N\chi} : \F'_{N\chi}\to \F_{N\chi}$ is surjective.  Since $\F_\chi^N\simeq \F_{N\chi}$, it follows that $\alpha_{\chi}$ is surjective. Since a surjective morphism of invertible sheaves is an isomorphism, it follows that $\alpha$ and consequently $\phi_X$ are isomorphisms.
\end{remark}

If $S=s$ is a point, then a family of affine stable toric $T$-varieties over $\X$ is simply a morphism $p_{\X}$ from affine stable toric $T$-variety $X$ to $\X$. The image of $p_{\X}$ is a closed $T$-subvariety $Z\subset\X$ which satisfy the following conditions:

\begin{itemize}
\item $k[Z]$ is multiplicity-free as a $T$-module;
\item If $Z=\cup Z_i$ is decomposition into irreducible components, then $\sigma_{\X}=\cup \sigma_{Z_i}$.
\end{itemize}

Conversely, if $Z\subset\X$ satisfy the conditions above, then there is unique affine stable toric $T$-variety $X$ which permits a surjective morphism to $Z$. Moreover, this morphism is defined uniquely modulo the finite subgroup of the $T$-equivariant automorphisms of $X$. We give a precise description of this subgroup in Remark~\ref{autg}. As we will show later, $\M_{\X, T}$ is a stack. If all $T$-subvarieties in $\X$, satisfying the conditions above, have Hilbert function $h_{\Sigma_{\X}}$, then the coarse moduli space of $\M_{\X, T}$ is $H_{X,T}$.

%\begin{example}

\medskip

Let $F: \M_{\X, T}\to {\rm Sch}$ be the functor that associates to any family $\{X\to S\}\in \M_{\X, T}$ its base $S$.  We denote by $\M_{\X, T}(S)$ the category of affine stable toric $T$-varieties over $\X$ with  base $S$. For simplicity we write $\M_{\X, T}(R)$ for $\M_{\X, T}(\Sp(R))$, where $R$ is some $k$-algebra.
By Remark \ref{rmk}, it follows that $\M_{\X, T}$ is a category fibred  in groupoids. Moreover, by the property of affine morphisms that descent data is effective \cite[Theorem~2.1]{SGA}, it follows that $F$ is a sheaf in categories.

\begin{remark}\label{autg}
 We shall describe the group of automorphisms of an affine stable toric $T$-variety $Y\in \M_{\X, T}(k)$. Note that the automorphisms of $Y$ correspond to $T$-equivariant automorphisms of the algebra $k[Y]$ that are trivial on the image of $k[\X]$. Such an automorphism is given by a map $\psi: \widehat{\Sigma}_{\X}\to k^{\times}$ such that the restriction of $\psi$ to the weight monoid $\Sigma_{Z}$ of any irreducible component $Z$ of $Y$ is a homomorphism of monoids. By Lemma \ref{lem2} (1), it follows that $\psi$ is uniquely determined by its restriction to $A\subset \widehat{\Sigma}_{\X}$.
Denote $\T=(k^{\times})^d$. The embedding $A\subset \widehat{\Sigma}_{\X}$ defines a surjective homomorphism of lattices:  $$\alpha: \Chi(\T)=\Z^d\to \Chi(T).$$
So we see that the group of automorphisms of $Y$ is the subgroup $\Gamma_Y$ of $\T$ defined by the following conditions:

\noindent 1. The characters lying in the kernel of the restriction of $\alpha$ to $\Z^{\#(A\cap \Sigma_{Z})}$, where $Z$ is an irreducible component of $Y$, are trivial on $\Gamma_Y$ (in particular, if $Y$ is irreducible, then this condition means that $\Gamma_Y\subset T\subset \T$).

\noindent 2. If $\chi\in \Sigma_{\X}$ is such that the image of $k[\X]_\chi$ in $k[Y]_\chi$ is non-zero, then the characters in $\alpha^{-1}(\chi)$ are trivial on $\Gamma_Y$.

\noindent In particular, by Lemma \ref{lem2} (2), it follows that for any $\chi\in \Chi(\T)$ the character $N\chi$ is trivial on $\Gamma_Y$, so $\Gamma_Y$ is finite.

In the same way,  the group of automorphisms of a family $X\in \M_{\X, T}(S)$ is isomorphic to a subgroup $\Gamma_X\subset \text{Mor}(S,\T)$ consisting of $S$-points of $\T$ such that for any geometrical $s\in S$ the corresponding $k$-point
of $\T$ is contained in $\Gamma_{X_{s}}\subset \T$.
\end{remark}

%So we see that the group of automorphisms of $X$ is a subgroup of $\G^A(S)$ consisting of $\lambda: \Z^A\to \O_S(S)^*$ satisfying the following conditions:

%\noindent 1. For any point $\overline s\in S$ the kernel of the restriction of the corresponding map ²$\to \O_S(S)^*\to k(\overline s)$ on any weigth monoid $\Chi^{T(k(\overline s))}_Y$, where $Y$ is an irreducible component of the fiber $X_{\overline s}$, is a homomorphism of monoids.

%\noindent 2. If $\chi\in \Sigma$ is such that the image of $k[\X]_\chi$ in $F_\chi$ is non-sero, then the caracters in $\alpha^{-1}(\chi)$ are trivial on $\Gamma$.

%If we have a map  $\psi: A\to \O_S(S)^*$, then $\psi$ defines a $T$-equivariant automorphisms of the sheaf of $\O_S$-algebras $F$ if and only if for any $\overline s\in S$ the corresponding map $\psi_{\overline s}: A\to k(\overline s)$ defines a $T$-equivaraiant  automorphisms of $k[X_{\overline s}]$. This is equivalent to the following conditions:

One can prove that for any families $X_1,X_2\in \M_{\X, T}(S)$ the functor $$\Isom(X_1, X_2): \rm{ Sch}/S\to \rm Set$$ that associates to any scheme $S'\to S$ the set of isomorphisms in $\M_{\X, T}(S')$ of $X_1\times_S S'$ to $X_2\times_S S'$, is an \'etale sheaf and, consequently, $\M_{\X, T}$ is a stack. But we shall show directly that
$\M_{\X, T}$ is a quotient stack.

Define the action of $T$ on $\X\times \A^d$ by $t\cdot (x,x_1,\ldots,x_d)=(t\cdot x, \chi_1(t) x_1, \ldots, \chi_d(t) x_d)$. Let $H:=H_{\X\times \A^d, T}$ be the corresponding toric Hilbert scheme. Note that there is the action of $\T:=(k^{\times})^d$ on $H$ induced by the diagonal action of $\T$ on $\A^d$.

Denote $U:=U_{\X\times \A^d, T}\subset H\times \X \times \A^d$ and let $p_H$ denote the projection on $H$. The inclusion $\iota: U\to H\times \X \times \A^d$ corresponds to the homomorphism $\Oo_H\otimes k[\X]\otimes k[\A^d] \to \G:=(p_H)_*(\Oo_U)$ of the $\Oo_H$-algebras.

Consider the open subscheme $\widetilde H \subset H$ defined by the following conditions:
\begin{enumerate}

\item The fibers of $p_H$ over $\widetilde H$ are reduced.

\item  The image of $k[\X]_{N\chi}$ generates the invertible sheaf $\G_{N\chi}$ at the points of $\widetilde H$ for any $\chi\in A$.

\item  The image of $x_i\in k[\A^d]$ trivializes the invertible sheaf $\G_{\chi_i}$ over $\widetilde H$ for any $\chi_i\in A$.
\end{enumerate}
It is clear that conditions (2) and (3) are open. Since $p_H$ is a good quotient it maps closed $T$-invariant subsets to closed subsets (see \cite[Theorem~2.3.6]{ADHL}). Now using \cite[Theorem~12.1.1]{EGA} one can show that (1) also is an open condition.

Note that $\widetilde H$ is invariant under the action of $\T$ on $H$. Denote by $\widetilde{\mathcal{H}}$ the corresponding open subfunctor.

\begin{theorem}\label{main}
There is an equivalence of the categories fibred in groupoids $\M^T_\X \cong [\widetilde H/\T]$.
\end{theorem}

\begin{proof}
 First, we construct a functor $\Upsilon:\M_{\X,T}\to [\widetilde H/\T]$. Let $X\in \M_{\X,T}(S)$. Consider
$$P:= \lsp_S \bigoplus_{r\in \Z^d} \bigotimes_{\chi_i\in A} \F_{\chi_i}^{r_i}\stackrel{q}{\to}S,$$ where $\F:=(p_S)_*(\Oo_X)$.
Note that $q$ is a principal fiber bundle under $\T$. Consider the pullback $$X_P:=X\times_S P\stackrel{p}{\to}P$$ of $X\to S$ by $q$. Then for any $\chi_j\in A$ we have a natural trivialization of $p_*(\Oo_{X_P})_{\chi_j}$:
$$p_*(\Oo_{X_P})_{\chi_j}=\left(\F_{\chi_j}\otimes_{\Oo_S}\bigoplus_{r\in \Z^d} \bigotimes_{\chi_i\in A} \F_{\chi_i}^{r_i}\right)^{\widetilde{\,\,\,\,\,\,\,\,}}\stackrel{\varkappa_j}{\xrightarrow{\sim}}\Oo_P.$$ (Here we use the notations of \cite[Exercise~II.5.17(e)]{Har}.) So, there is a natural surjection of $\Oo_P$-algebras $\Oo_P\otimes k[\X]\otimes k[\A^d]\stackrel{\varphi}{\twoheadrightarrow} p_*(\Oo_{X_P})$, where $\varphi(1\otimes 1\otimes x_i)=\varkappa_i^{-1}(1)$, which, in turn, corresponds to the closed immersion $X_P\hookrightarrow P\times \X \times \A^d$ of their local spectrums.

We state that $X_P\in\widetilde{\mathcal{H}}(P)$. We need to check conditions (1)---(3) for $p$. Condition (1) is satisfied already for $q$ and it is stable under the base change. By Remark~\ref{fin}, condition (2) is equivalent to the property of morphism $p$ to be finite that is also satisfied for $q$ and stable under the base change. Condition (3) follows from the definition of $\varphi$.

 An element $X_P\in\widetilde{\mathcal{H}}(P)$ corresponds to a morphism $P\to \widetilde H$. It is clear that this morphism is $\T$-equivariant. Thus, we have the diagram
$$
\xymatrix{
P \ar[r]\ar[d]_q & \widetilde H \\
\,\,\, S \,\, ,
}
$$
which is an element of $[\widetilde H/\T](S)$.

Notice that we have an action of $\T$ on $X_P\subset P\times \X \times \A^d$ given by the restriction of the diagonal action of $\T$ (where $\T$ acts trivially on $\X$).

Since $$(q\circ p)_*(\Oo_{X_P})^{\T}=\F\otimes_{\Oo_S}q_*(\Oo_P)^{\T}\simeq \F,$$ we see that
 $p_S: X\to S$ is canonically isomorphic (in $\M_{\X,T}(S)$) to the quotient of $p$ : $$p/\!/\T:X_P/\!/\T\simeq X\to S.$$

Conversely, given  a principal fiber bundle $q:P\to S$ under $\T$  and a $\T$-equivariant morphism $P\to \widetilde H$, consider the fibred product $X_P:=P\times_{\widetilde H}\widetilde U\in \widetilde{\mathcal{H}}(P)$, where $\widetilde U$ is the universal family over $\widetilde H$. Then we have the (flat) quotient morphism $X:=X_P/\!/\T\to P/\!/\T=S$. Since the projection $X_P\to \X\times P$ is finite (use Remark~\ref{fin} again), it follows that the corresponding morphism $X\to \X\times S$ is finite and $X\in \M_{\X, T}(S)$. This gives us a functor $\Upsilon': [\widetilde H/\T]\to\M_{\X, T}$.

We see immediately that $\Upsilon'\circ\Upsilon\simeq {\rm Id}_{\M_{\X, T}}$ and $\Upsilon\circ\Upsilon'\simeq {\rm Id}_{[\widetilde H/\T]}$.
\end{proof}

\begin{example}
Let $T=k^{\times}$, $\X=\A^1$ and $T$ acts on $\X$ by $t\cdot x=t^m\cdot x$, $m>0$. Then $\Chi(T)\cong \Z$, $\Sigma_{\X}=m\N$, $\widehat{\Sigma}_{\X}=\N$. Take $A=\{1\}$, $d=1$, $N=m$. It is easy to see that they satisfy the conditions of Lemma~\ref{lem2}. The diagonal action of $T$ on the affine plane $\A^2=\X\times\A^1$ is given by the matrix $\bigl(\begin{smallmatrix}
t^d & 0\\0 & t
\end{smallmatrix}\bigr)$. We are in situation of the subtorus action on the toric variety, so one can use \cite[Theorem~4.5]{C} for the computation of fans of $H$ and $U$ (see the picture below), but it is also not hard to see directly that $H=\P^1$ and $U$ is the weighted blow-up of $\A^2$ with weights $(1,m)$.

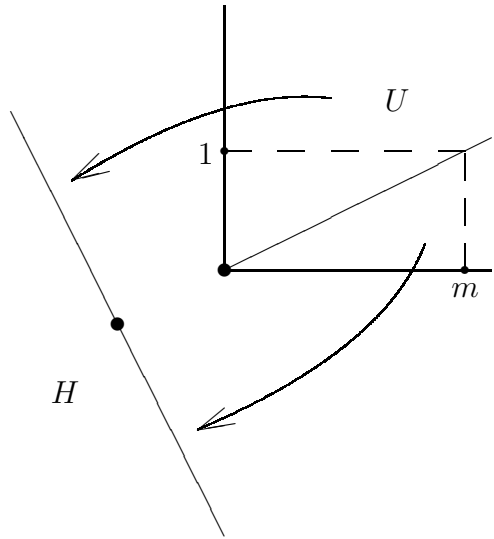
\begin{center}
\begin{picture}(200,200)

\put(100,100){\line(0,1){100}}
\put(100,100){\line(1,0){100}}
\put(100,100){\line(2,1){100}}
\put(100,100){\circle*{5}}
\put(60,80){\circle*{5}}
\put(100,145){\circle*{3}}
\multiput(100,145)(20,0){5}
{\line(1,0){10}}
\put(190,100){\circle*{3}}
\multiput(190,100)(0,17){3}
{\line(0,1){10}}
\put(60,80){\line(-1,2){40}}
\put(60,80){\line(1,-2){40}}
\qbezier(175,110)(160,70)(90,40)
\qbezier(140,165)(100,170)(43,134)
\put(43,134){\line(1,1){10}}
\put(43,134){\line(4,1){14}}
\put(90,40){\line(6,5){10}}
\put(90,40){\line(6,1){14}}

\put(35,50){$H$}
\put(160,160){$U$}
\put(185,90){$m$}
\put(90,140){$1$}
\end{picture}
\end{center}

Let $(x,y)$ and $(u:v)$ be coordinates on $\A^2$ and homogeneous coordinates on $H$ respectively. The torus $\T=k^{\times}$ acts on $H$ by $t\cdot (u:v)=(u:t^mv)$. The universal family $U\subset H\times \A^2$ is the surface $xu=y^mv$. We see immediately that the fiber of $p_H$ over $(0:1)$ is not reduced and $k[\X]_N=\langle x\rangle$ does not generate $\G_{N}$ at the point $(1:0)$. So $\widetilde{H}=\P^1\setminus\{(0:1),(1:0)\}$ and $[\widetilde{H}/\T]\cong[pt/(\Z/m\Z)]$.

 On the other side, $X=\A^1$ considered as a toric variety under the torus $T$ is the unique affine stable toric $T$-variety with $\Sigma_X=\N$. Then one can show that all families of affine stable toric varieties over $\X$ are trivial. The group $\Z/m\Z$ is exactly the group of automorphisms of $X\in \M_{\X,T}(k)$ (see Remark~\ref{autg}), so $\M_{\X,T}\cong[pt/(\Z/m\Z)]$. That fits nicely with the statement of Theorem~\ref{main}.
\end{example}

\section*{Acknowledgements}
Authors are grateful to Ivan V. Arzhantsev for many helpful comments and suggestions. The first author is thankful to Michel Brion for stimulating ideas and useful
 discussions. The second author thanks Dmitriy O. Orlov and Hendrik S\"{u}{\ss} for helpful discussions and comments.

\normalsize
%%%%%%%%%%%%%%%%%%%%%%%%%%%%%%%%%%%%%%%

\bigskip

\end{document}